\documentclass[12pt]{amsart}
\usepackage{latexsym, amsmath, amscd, amssymb, amsthm} 
\newtheorem{te}{Theorem}

\begin{document}

 \title{   Analogue  of Sylvester-Cayley formula  for invariants of ternary form} 

\author{Leonid Bedratyuk} \address{ Khmel'nyts'ky National University, Instytuts'ka st. 11, Khmel'nyts'ky , 29016, Ukraine}
\email {bedratyuk@ief.tup.km.ua}
\begin{abstract}
The number $\nu_d(n)$ of linearly independed homogeneous invariants of  degree $n$  for the ternary form of degree $d$  is calculated. The folloving formula is  hold
$$
\nu_d(n)=c_d{(n,0,0)}+c_d{(n,3,0)}+c_d{(n,0,3)}-2\, c_d{(n,1,1)}-c_d{(n,2,2)},
$$
here $c_d(n,i,j)$  is number  of nonnegative integer solutions of the system of equations 
$$
\left\{
 \begin{array}{c}
\displaystyle \sum_i i\,\alpha_{i,j}=\frac{n\,d}{3}-\frac{i-j}{3},\\
\displaystyle \sum_j  j\,\alpha_{i,j}=\frac{n\,d}{3}-\frac{i+2\,j}{3},\\
\displaystyle \sum_{i,j} \alpha_{i,j}=n.
\end{array}
\right.
$$
Also, $\nu_d(n)$  is  equal to coefficient of $\displaystyle (p q)^{\frac{n\,d}{3}}$   in the following  polynomial expression 
$$
(1+p\,q+\frac{q^2}{p}-2\,q-q^2) \left(\sum_{i_1+i_2+\cdots +i_d\leq n} \left[ \begin{array}{c} 1 \\ i_1  \end{array} \right]_{pq}\left[ \begin{array}{c} 2 \\ i_2  \end{array} \right]_{pq}\cdots \left[ \begin{array}{c} d \\ i_d  \end{array} \right]_{pq} \right),
$$
here $\left[ \begin{array}{c} d \\ k  \end{array} \right]_{pq}$  is $pq$-binomial coefficient 
$$
\left[ \begin{array}{c} d \\ k  \end{array} \right]_{pq}:=\frac{(p^{d+1}-q^{d+1}) \ldots (p^{d+k}-q^{d+k}) }{(p-q) (p^{2}-q^{2}) \ldots (p^{k}-q^{k})}.
$$
\end{abstract}

\maketitle

\noindent
{\bf 1.}
Let $T_d$ be the  vector $\mathbb{C}$-space of  ternary  forms of degree $d:$ 
$$ 
u(x,y,z)=\sum_{i+j\leq d} \, \frac{d!}{i! j! (d{-}(i+j))!}a_{i, j}\, x^{d-(i+j)} y^i z^j,
$$
where $a_{i, j} \in \mathbb{C}.$
The vector  space $T_d$  is endowed with the natural action of the group  $SL(3,\mathbb{C}),$  and  with the corresponding action of the Lie algebra $\mathfrak{sl_{3}}.$
Let us identify the algebra of polynomial function  $\mathbb{C}[T_d]$    with the polynomial $\mathbb{C}$-algebra $A_d$ of the variables $a_{0\,0},a_{1,0},\ldots, a_{0,d}.$
The natural action of the group $SL(3,\mathbb{C})$  on $T_d$  induces    actions  of   $SL(3,\mathbb{C})$  and  $\mathfrak{sl_{3}}$   on the algebra $A_d.$  The corresponding invariant ring   $A_d^{SL_3}=A_d^{\mathfrak{sl_{3}}}$ is  called the ring of invariants for the ternary form of degree  $d.$

The ring  $A_d^{SL_3}$  is  graded ring 
$$
A_d^{SL_3}=(A_d^{SL_3})_0+(A_d^{SL_3})_1+\cdots+(A_d^{SL_3})_n+ \cdots,
$$
here  $(A_d^{SL_3})_n$  is the vector space generated  by   homogeneous  invariants of degree $n.$
In this  paper  we discover  the exact formula  for calculation the dimension of the vector space  $(R_d^{SL_3})_n.$ The formula  is analogue  of well-known Sylvester-Cayley formula  for invariants of binary form

\noindent
{\bf 2.} To begin with, we give a short proof of Sylvester-Cayley fo invariants of   binary form. We  will use  the idea of the proof in next section    to derive  an analogous   result in the  case  of ternary form.

Let  $V\cong \mathbb{C}^2$ be standard  two-dimensional representation of Lie algebra  $\mathfrak{sl_{2}}.$ The irreducible representation   $V_d=\langle v_0,v_1,...,v_d \rangle,$ $\dim V_d=d+1$ is the symmetric $d$-power of the standard representation  $V=V_1,$  i.e. $V_d=S^d(V),$  $V_0 \cong  \mathbb{C}.$  
The basis elements    $\hat D=\Bigl( \begin{array}{ll}  0\, 1 \\ 0\,0 \end{array} \Bigr),$ $D= \Bigl( \begin{array}{ll}  0\, 0 \\ 1\,0 \end{array} \Bigr)$, $E= \Bigl( \begin{array}{ll}  1 &  \phantom{-}0 \\  0 &-1 \end{array} \Bigr)$ of the algebra    $\mathfrak{sl_{2}}$ act on    $V_d$  by the derivations : 
$$
D(v_i)=i\, v_{i-1},  \hat D(v_i)=(d-i)\,v_{i+1}, E(v_i)=(d-2\,i)\,v_i.
$$
The action of   $\mathfrak{sl_{2}}$  is extended to action on the symmetrical algebra  $S(V_d)$ in the natural way. The algebra  $I_d$
$$
I_d= \displaystyle{ S(V_d)^{\mathfrak{sl_{2}}}}=\{ v \in S(V_d)|  D(v)=0, \hat D(v) =0 \},
$$
is called the algebra of invariant of the "binary form"  of degree $d.$

The algebra  $S(V_d)$  is graded 
$$
S(V_d)=S^0(V_d)+S^1(V_d)+\cdots +S^n(V_d)+\cdots,
$$
and each   $S^n(V_d)$ is complete reducibly
 respresentation of the Lie algebra  $\mathfrak{sl_{2}}.$
Thus, the following decomposition is hold
$$
S^n(V_d) \cong \gamma_d(n,0) V_0+\gamma_d(n,1) V_1+ \cdots +\gamma_d(n,d\,n) V_{d\,n},  \eqno{(*)}
$$
here  $\gamma_d(n,i)$ is  the  multiplicities of the representation  $V_i$  in the decomposition of  $S^n(V_d).$ In particular,  the multiplicities   $\gamma_d(n,0)$  of trivial respresentation  $V_0$ is  equal to number of linearly independed homogeneous invariants of  degree $n$ for the  binary form of degree  $d.$

The set of weights (the eigenvalues of the operator $E$) of a representatin  $W$ denote by  $\Lambda_{W},$  in particular, $\Lambda_{V_d}=\{-d, -d+2, \ldots, d \}.$ 

 The formal sum 
$$
{\rm Char}(W)=\sum_{i \in \Lambda_{W}} n_W(i) q^i,
$$
is called the character   of the respresentation  $W,$  
here   $n_W(i)$ denotes the   multiplicities  of the weight $i \in \Lambda_{W}.$
Since, the   multiplicities of any weight of the irreducible respresentation $V_d$  is  equal to 1, we have  
$$
{\rm Char}(V_d)=q^{-d}+q^{-d+2}+\cdots+q^{d}=\frac{q^{d+1}-q^{-(d+1)}}{q-q^{-1}}.
$$

The  character   $ {\rm Char}(S^n(V_d))$ of the representation  $S^n(V_d)$  is equal to ð³âíèé  $H_d(q^{-d},q^{-d+2},\ldots,q^{d}),$ see \cite{FH},  where   $H_d(x_0,x_1,\ldots,x_d)$ is  the complete symmetrical function    $$H_d(x_0,x_1,\ldots,x_d)=\sum_{|\alpha|=n} x_0^{\alpha_0} x_1^{\alpha_1} \ldots x_d^{\alpha_d} , |\alpha|=\sum_i \alpha_i,$$
see \cite{Mc}  for details.

By putting    $x_i=p^{d-2\,i},$ $i=0,\ldots, d,$  we  obtain the expression for the character of   ${\rm Char}(S^n(V_d)):$ 
$$
{\rm Char}(S^n(V_d))= \sum_{|\alpha|=n} (p^d)^{\alpha_0} (p^{d-2\cdot 1})^{\alpha_1} \ldots (p^{d-2\,d})^{\alpha_d} = \sum_{|\alpha|=n} p^{d\,n-2 (\alpha_1+2\alpha_2+\cdots + d\, \alpha_d)}=
$$
$$
=\sum_{i=0}^{d\,n} \omega_d(n,i) p^{d\,n-2\,i},
$$
here   $\omega_d(n,i) $  is the number nonnegative integer solutions of the equation $$\alpha_1+2\alpha_2+\cdots + d\, \alpha_d=\displaystyle \frac{d\,n-i}{2}$$  on the assumption that  $ |\alpha|=n.$ In particular, the coefficient of $q^0$ (the     multiplicities  of zero weight ) is  equal to  $ \omega_d(n,\frac{d\,n}{2}),$ and the coefficient of  $q^2$ is  equal  $ \omega_d(n,\frac{d\,n}{2}-1).$

On the other hand, the decomposition $(*)$ implies the  equality
$$
{\rm Char}(S^n(V_d))=\gamma_d(n,0) {\rm Char}(V_0) +\gamma_d(n,1) {\rm Char}(V_1)+ \cdots +\gamma_d(n, d\,n) {\rm Char}(V_{d\,n}).
$$
Now  we  can prove 
\begin{te}[Sylvester-Cayley]  $$\gamma_d(n,0)=\omega_d(n,\frac{d\,n}{2}) - \omega_d(n,\frac{d\,n}{2}-1).$$
\end{te}
\begin{proof}
The zero weight occurs  once  in any irreducible representation $V_i,$  $i=0,\ldots, d$ therefore 
$$
 \omega_d(n,\frac{d\,n}{2})=\gamma_d(n,0) +\gamma_d(n,1)+ \cdots +\gamma_d(n, d\,n).
$$
The weight   $2$ occurs once in any  irreducible representation (except trivial)   $V_i,$  $i=1,\ldots, d,$ therefore 
$$
 \omega_d(n,\frac{d\,n}{2}-1)=\gamma_d(n,1) +\gamma_d(n,2)+ \cdots +\gamma_d(n, d\,n).
$$
Thus 
$$
\gamma_d(n,0)= \omega_d(n,\frac{d\,n}{2})- \omega_d(n,\frac{d\,n}{2}-1).
$$
This completes the proof.
\end{proof}

One may show,  see \cite{Hilb}, that 
$$
\gamma_d(n,0)=\left((1-q) \left[ \frac{d}{n} \right]_q\right)_{q^{\frac{n\,d}{2}}},
$$
here  $\left[ \begin{array}{c} d \\ n  \end{array} \right]_q$ is  the  $q$-binomial coefficient:
$$
\left[ \begin{array}{c} d \\ n  \end{array} \right]_q:=\frac{(1-q^{d+1})(1-q^{d+2})\ldots(1-q^{d+n})}{(1-q)(1-q^{2})\ldots(1-q^{n})}.
$$

\noindent
{\bf 3.} In the Lie algebra   $\mathfrak{sl_{3}}$
denote  by  $E_{i\,j}$ the matrix unities.The matrix³ $H_1:=E_{1\,1}{-}E_{2\, 2}$  and  $H_2:=E_{2\, 2}{-}E_{3\,3}$ generate the Cartan  subalgebra in  $\mathfrak{sl_{3}}.$The elements   $H_1, H_2$  act on the  $A_d,$   see \cite{A} for the proof,  in the following way
$$
H_1(a_{i,j})=(d-(2i+j)) a_{i,j},  H_2(a_{i,j})=(i-j) a_{i,j}.
$$

Every monomial  $a^{\alpha}:=a_{0,0}^{\alpha_{0,0}} a_{1,0}^{\alpha_{1,0}} \cdots a_{0,d}^{\alpha_{0,d}}$ of degree $n$ is the eigenvector of the operators  $H_1,$ $ H_2$ with the eigenvalues   $n\,d- (2 \omega_1(\alpha)+\omega_2(\alpha))$  and  $\omega_1(\alpha)-\omega_2(\alpha) $ correspondingly. Here $\omega_1(\alpha)=\sum_{i} i\, \alpha_{i,j},$   $\omega_2(\alpha)=\sum_{j} j\, \alpha_{i,j},$  $|\alpha|:=\sum_{i,j} \alpha_{i,j}=n.$

Denote by $\Gamma_{m_1,\,m_2}$ the standard irreducible representation of   $\mathfrak{sl_{3}}$ with the  highest  weight $\lambda=(m_1,m_2).$  It can easily be checked that  $\mathbb{C}^3 \cong \Gamma_{1,0},$  $A_d \cong (S^d(\Gamma_{1,0}))^{*}  \cong \Gamma_{0,d},$  see details in  \cite{FH}.

Let us recall the definition of the formal character of the representation of the Lie  algebra $\mathfrak{sl_{3}}$. Let  $\Lambda$  be  a weight lattice of all finitedimension  representation of   $\mathfrak{sl_{3}},$ and   $\mathbb{Z}(\Lambda)$  be  their group ring. The ring $\mathbb{Z}(\Lambda)$ is free $\mathbb{Z}$-module with the basis elements $e(\lambda),$ $\lambda=(\lambda_1,\lambda_2) \in \Lambda,$  $e(\lambda) e(\mu)=e(\lambda+\mu),$  $e(0)=1.$
Let  $\Lambda_{\lambda}$ be   set of all weights  of the representation   $\Gamma_{\lambda}.$  Then  the formal character  ${\rm Char}(\Gamma_{\lambda})$ is  defined, see  \cite{Hum},  as formal sum  $\sum_{\mu \in \Lambda_{\lambda}} n_{\lambda}(\mu) e(\mu) \in \mathbb{Z}(\Lambda),$ here  $n_{\lambda}(\mu)$ is the   multiplicities  of the weight  $\mu$   in the representation  $\Gamma_{\lambda}.$  For example, for the highest weight  $\lambda=(1,1)$ we  have   $${\Lambda_{(1,1)}=\{(1,1),(-1,2), (1,-2), (0,0), (-2,1),(-1,-1) \}}, $$  therefore
$$
{\rm Char}(\Gamma_{(1,1)})=e(1,1)+e(-1,2)+e (1,-2)+2\, e(0,0)+e (-2,1)+e(-1,-1).
$$

The basis elements  $a_{i,j}$ of the ring   $A_d$  have weights  $(d-(2\,i+j), i-j),$  $i+j \leq d,$  of the the   multiplicities $1,$ thus 
$$
{\rm Char}(\Gamma_{0,d})=\sum_{i+j \leq d} e(d-(2\,i+j), i-j).
$$
The characted of the symmetric  power  $S^n(\Gamma_{d,\,0})$ is the  complete symmetric polynomial of degree $n$  of the variable  set    $e(d-(2\,i+j), i-j),$ $i+j \leq d,$ see \cite{FH}. Therefore we  have  
$$
{\rm Char}(S^n(\Gamma_{0,d}))=\sum_{|\alpha|=n} e(0,0)^{\alpha_{0,0}} e(1,0)^{\alpha_{1,0}} \cdots e(0,d)^{\alpha_{d,0}}, |\alpha|=\sum_{i,j} \alpha_{i,j}.
$$
 After  simplification we obtain   
$$
{\rm Char}(S^n(\Gamma_{0,d}))=\sum_{|\alpha|=n} e(n\,d -2\, \omega_1(\alpha)-\omega_2(\alpha), \omega_1(\alpha)-\omega_2(\alpha))=\sum_{(i,j)  \in \Lambda_{(nd,0)}} c_d{(n,i,j)}(n) e(i,j),
$$
here  $c_d(n,i,j)$  is the number    of nonnegative integer solutions of the system of equations 
 
$$
\left \{
\begin{array}{l}
2 \omega_1(\alpha)+\omega_2(\alpha)=d\,n-i, \\
\omega_1(\alpha)-\omega_2(\alpha)=j, \\
|\alpha|=n,
\end{array}
\right.
$$
or  the equivalent system
$$
\left \{
\begin{array}{l}
 \omega_1(\alpha)=\displaystyle \frac{d\,n}{3}-\frac{i-j}{3}, \\
\omega_2(\alpha)=\displaystyle \frac{d\,n}{3}-\frac{i+2j}{3}, \\
|\alpha|=n.
\end{array}
\right.
\eqno{(**)}
$$

For  $i=j=0$   we  get  $\omega_1(\alpha)=\omega_2(\alpha)$ and  $n\,d=3 \omega_1(\alpha).$  It implies   ${n\,d =0\mod 3.}$ By adding first two equations we  obtain   $ i-j=0 \mod 3.$  It implies   $i+2\,j=0 \mod 3.$  

On the representation   $\Gamma_{\lambda}$  let us define  the value $E_{\lambda}$ in the following way
$$
E_{\lambda}=n_{\lambda}(0,0)+n_{\lambda}(3,0)+n_{\lambda}(0,3)-2\,n_{\lambda}(1,1)-n_{\lambda}(2,2).
$$
We  assume that  $n_{\lambda}(i,j)=0 $ if    $ (i,j) \notin \Lambda_{\lambda}.$ 

The following theorem play crucial  role in the calculation.
\begin{te}
$$
E_{\lambda}=\left\{ 
\begin{array}{c}
1, \lambda=(0,0) ,\\
0,  \lambda \neq (0,0).
\end{array}
\right.
$$
\end{te}
\begin{proof}

For   $\lambda=(0,0)$ the statement  is obvious.

Suppose  now  $\lambda \neq (0,0).$
 Let us recall   that  the weight multiplicities of $\Gamma_{(i,j) }$ increase by one on each 
of the concentric hexagons of the eigenvalue lattice  and are constant on the internal 
triangles.  To prove the theorem let us consider the cases - $i-j=0,$  $|i-j| >3$  and $|i-j|=3.$

Let   $\lambda=(m,m),$ $m>0.$  Then  $n_{\lambda}(0,0)=m+1,$   and  it is clear that    $n_{\lambda}(1,1)=m$ and 
 $n_{\lambda}(3,0)=n_{\lambda}(2,2)=n_{\lambda}(0,3)=m-1.$  
It follows that $$E_{\lambda}{=}m+1+m-1+m-1-2\,m -(m-1)=0.$$

Let  $\lambda=(m,k),$ $|m-k| >3.$  Then all multiplicities   $n_{\lambda}(0,0),$   $n_{\lambda}(3,0),$ $n_{\lambda}(1,1),$ $n_{\lambda}(2,2),$ $n_{\lambda}(0,3)$ are equal to  $\min(m,k)+1$ and therefore $E_{\lambda}=0.$  
 
Let  $\lambda=(m,k),$ $m-k =3.$  Then    $n_{\lambda}(0,0)=n_{\lambda}(3,0)=n_{\lambda}(1,1)=k+1$ and  ${n_{\lambda}(2,2)=k}$, $n_{\lambda}(0,3)=k,$  çâ³äêè $E_{\lambda}=0.$ 

If    $\lambda=(m,k),$ $k-m =3$  then  $n_{\lambda}(0,0)=n_{\lambda}(0,3)=n_{\lambda}(1,1)=m+1$ and $n_{\lambda}(2,2)=m,$ $n_{\lambda}(3,0)=m.$  This yields that $E_{\lambda}=0.$ 
\end{proof}

Now we  are  ready to calculate the value $\nu_d(n).$
\begin{te}
The number  $\nu_d(n)$ of linearly independed homogeneous invariants of  degree $n$  for the ternary form of degree $d$  is  equal 
$$
\nu_d(n)=c_d{(n,0,0)}+c_d{(n,3,0)}+c_d{(n,0,3)}-2\, c_d{(n,1,1)}-c_d{(n,2,2)}.
$$
\end{te}
\begin{proof}
The  number  $\nu_d(n)$  is equal to the multiplicities  $\gamma_d(0,0)$ of trivial representation  $\Gamma_{0,0}$  in the symmetrical power  $S^n(\Gamma_{0,\,d}).$   Suppose   the decomposition is hold
$$
S^n(\Gamma_{0,d})=\gamma_d(0,0) \Gamma_{0,0}+\cdots +\gamma_d(0,n\,d) \Gamma_{0,n\,d}=\sum_{\lambda \in \Lambda_{(0,n\,d)}^{+}} \Gamma_{\lambda}.
$$
It is implies, 
$$
{\rm Char}(S^n(\Gamma_{0,d}))=\gamma_d(0,0) {\rm Char}(\Gamma_{0,0})+\cdots +\gamma_d(0,n\,d) {\rm Char}(\Gamma_{0,n\,d}).
$$
Thus, 
$$
\sum_{(i,j)  \in \Lambda_{(0,n\,d)}} c_d{(n,i,j)}(n) e(i,j)=\sum_{\lambda }\gamma_d(\lambda) {\rm Char}(\Gamma_{\lambda}){=}\sum_{(i,j)  \in \Lambda_{(0,n\,d)}} \sum_{\lambda} \gamma_d(\lambda) n_{\lambda}(i,j) e(i,j) .
$$
 It follows that $c_d{(n,i,j)}=\sum_{\lambda} \gamma_d(\lambda) n_{\lambda}(i,j).$  By using previous theorem we have
$$
c_n{(n,0,0)}+c_d{(n,3,0)}+c_d{(n,0,3)}-2\, c_d{(n,1,1)}-c_d{(n,2,2)}=\sum_{\lambda} \gamma_d(\lambda) E_{\lambda}=\gamma_d(0,0).
$$
Taking into account  $\gamma_d(0,0)=\nu_d(n)$ we  obtain   the statement  of the theorem. 
\end{proof}

{\bf 4.}  Let us derive the fourmula  for calculation of  $\nu_d(n).$  

It is not hard to prove that  the number  $c_d(n,0,0)$   of   nonnegative integer solutions of the following system  
$$
\left \{
\begin{array}{c}
\omega_1(\alpha)=\frac{d\,n}{3},\\
\omega_2(\alpha)=\frac{d\,n}{3},\\
|\alpha|=n
\end{array}
\right.
$$
is equal to coefficient of  $\displaystyle t^n (pq)^{\frac{d\,n}{3}} $   the expansion of the following series 
$$
R_{d}=(\prod_{k+l \leq  d } (1-t p^k q^l))^{-1}.
$$
Denote it in such way -- $c_d(n,0,0)=(R_d)_{t^n (pq)^{\frac{d\,n}{3}}}.$

The number  $c_d(n,3,0)$  of   nonnegative integer solutions of the following system
$$
\left \{
\begin{array}{c}
\omega_1(\alpha)=\frac{d\,n}{3}-1,\\
\omega_2(\alpha)=\frac{d\,n}{3}-1,\\
|\alpha|=n
\end{array}
\right.
$$
is  equal to  $(R_d)_{t^n (pq)^{\frac{d\,n}{3}-1}}=(p\,q R_d)_{t^n (pq)^{\frac{d\,n}{3}}}.$
Similarly,  
$$
c_d(n,0,3)=\left(\frac{q^2}{p} R_d\right)_{t^n (pq)^{\frac{d\,n}{3}}},  c_d(n,1,1)=\Bigl(q R_d\Bigr)_{t^n (pq)^{\frac{d\,n}{3}}}, c_d(n,2,2)=\Bigl(q^2 R_d\Bigr)_{t^n (pq)^{\frac{d\,n}{3}}}.
$$
Therefore the following formula is   hold 
$$
\nu_d(n)=\left((1+p\,q+\frac{q^2}{p}-2\,q-q^2) R_d\right)_{t^n (pq)^{\frac{d\,n}{3}}}.
$$

Let us define the  $pq$-binomial coefficients in the following way 
$$
\left[ \begin{array}{c} n \\ k  \end{array} \right]_{pq}:=\frac{(p^{n+1}-q^{n+1}) \ldots (p^{n+k}-q^{n+k}) }{(p-q) (p^{2}-q^{2}) \ldots (p^{k}-q^{k})}.
$$
The function  $G_m:=\left(\prod_{k+l=m}( 1-t p^k q^l)\right)^{-1}$ are generating functions for the $pq$-binomial coefficients:
$$
G_m=\prod_{k+l=m}( (1-t p^k q^l)^{-1}=1+\left[ \begin{array}{c} m \\ 1 \end{array} \right]_{pq} t+\left[ \begin{array}{c} m \\ 2  \end{array} \right]_{pq} t^2+\cdots
$$
Therefore,
$$
R_{d}=G_0 G_1 \cdots G_d=\sum_{n} \left(\sum_{i_1+i_2+\cdots +i_d\leq k} { 1 \choose i_1}_{pq}{ 2 \choose i_2}_{pq}\cdots { d \choose i_d}_{pq} \right) t^n.
$$
Thus, the  number  $\nu_d(n)$ of  linearly independed homogeneous invariants of  degree $n$  for the ternary form of degree $d$ is  calculating by the formula
$$
\nu_d(n)=\left((1+p\,q+\frac{q^2}{p}-2\,q-q^2) \sum_{i_1+i_2+\cdots +i_d\leq k} { 1 \choose i_1}_{pq}{ 2 \choose i_2}_{pq}\cdots { d \choose i_d}_{pq} \right)_{{(pq)^{\frac{d\,n}{3}}}}.
$$
Below  are present the first several  terms of Poincar\'e series for the rings of invariants of the ternary forms of degrees 3,4,5,6,7.
$$
\begin{array}{l}
P_3(t)=t^{4} + t^{6} + t^{8} + t^{10} + 2\,t^{12} + t^{14} + 2\,t^{16}
 + 2\,t^{18} + 2\,t^{20} + 2\,t^{22} + 3\,t^{24} + 2\,t^{26} + ..\\
P_4(t)=t^{3} + 2\,t^{6} + 4\,t^{9} + 7\,t^{12} + 11\,t^{15} + 19\,t^{18}
 + 29\,t^{21} + 44\,t^{24} + 67\,t^{27} + 98\,t^{30}+..\\
P_5(t)=2\,t^{6} + t^{9} + 19\,t^{12} + 24\,t^{15} + 178\,t^{18} + 383\,t
^{21} + 1470\,t^{24} + 3331\,t^{27} + 9381\,t^{30}+...\\
P_6(t)=t^{3} + t^{4} + t^{5} + 4\,t^{6} + 5\,t^{7} + 8\,t^{8} + 17\,t^{9
} + 28\,t^{10} + 48\,t^{11} + 99\,t^{12}+172\,t^{13}+..\\
P_7(t)=3\,t^{6} + 13\,t^{9} + 421\,t^{12} + 4992\,t^{15} + 60303\,t^{18}
 + 548966\,t^{21}+..
\end{array}
$$
The results in the case $d=4$  coincides completelly with the results of Shioda, see \cite{Shi}.

\end{document}